\documentclass{amsart}
\usepackage{amsfonts,amssymb,amscd,amsmath,enumerate,verbatim,calc}
\usepackage[all]{xy}

\newcommand{\wrt}{with respect to}

\newcommand{\n}{\mathfrak{n} }
\newcommand{\m}{\mathfrak{m} }

\newcommand{\xb}{\mathbf{x} }

\newcommand{\Sc}{\mathcal{S} }

\newcommand{\rt}{\rightarrow}

\newcommand{\wh}{\widehat }

\newcommand{\charp}{\operatorname{char}}

\newcommand{\Ass}{\operatorname{Ass}}

\theoremstyle{plain}

\newtheorem{theorem}{Theorem}[section]

\theoremstyle{definition}

\newtheorem{s}[theorem]{}

\theoremstyle{remark}

\begin{document}

\title[Associate primes]{Associate primes of local cohomology modules over certain quotients of regular rings}
\author{Tony~J.~Puthenpurakal}
\date{\today}
\address{Department of Mathematics, IIT Bombay, Powai, Mumbai 400 076, India}

\email{tputhen@gmail.com}
 \begin{abstract}
Let $R$ be a regular ring containing a field $k$.  Let $\xb = x_1, \ldots, x_r$ be a regular sequence in $R$ such that  $R/(\xb)$ is a regular ring. Fix $m \geq 1$. Set $A_m = R/(\xb)^m$.
We show that for any ideal $Q$ of $A_m$ the set $\Ass H^i_Q(A_m)$ is a finite set for $i \geq 0$, in the following cases:
\begin{enumerate}
  \item $\charp k = p > 0$.
  \item $\charp k = 0$, $R$ is local or a smooth affine algebra over $k$.
\end{enumerate}
\end{abstract}
 \maketitle
\section{introduction}
Let $S$ be a Noetherian ring and let $I$ be an ideal in $S$. Let $H^i_I(S)$ be the $i^{th}$-local cohomology module of $S$ \wrt \ $I$. The question whether $\Ass H^i_I(S)$ is a finite set was first raised by Huneke. If $S$ is a regular ring containing a field $k$ with $\charp k = p > 0$ then Huneke and Sharp proved that $\Ass H^i_I(S)$ is a finite set for all $i \geq 0$, see \cite{HS}. If $S$ is a regular ring containing a field $k$ of characteristic zero then Lyubeznik proved that $\Ass H^i_I(S)$ is a finite set  or all $i \geq 0$ if either $S$ is local or if $S$ is a smooth affine algebra over $k$; see \cite{Lyu-1}.
In  \cite{N}, it is  proved that
  if $S \rt R$ is a homomorphism of Noetherian rings that splits, then for every ideal $I$ in $S$ and every non-negative integer $i$, if $\Ass_R H^i_{IR}(R)$ is finite then $\Ass_S H^i_I(S)$ is finite. The first example of a ring $S$ such that $\Ass H^i_I(S)$ is an infinite set  was  constructed by Singh \cite{S}. In the paper \cite{SS}, the authors constructed  examples of rings with good properties such that $\Ass H^i_I(S)$ is infinite.
   In this paper we construct large examples of singular rings $S$ such that $\Ass H^i_I(S)$ is finite for all $i \geq 0$.

   \s \label{const} Let $R$ be a regular ring containing a field $k$.  Let $\xb = x_1, \ldots, x_r$ be a regular sequence in $R$ such that  $R/(\xb)$ is a regular ring. Fix $m \geq 1$. Set $A_m = R/(\xb)^m$.
We show the following result:
\begin{theorem}\label{main}
(with hypothesis as in \ref{const}. Let $I$ be an ideal in $R$. Let $P$ be a prime ideal in $R$ containing $\xb$. Fix $i \geq 0$ and $m \geq 1$. The following assertions are equivalent:
\begin{enumerate}[\rm (i)]
  \item  $P \in \Ass_R H^i_I(A_1)$.
  \item $P \in \Ass_R H^i_I(A_m)$.
\end{enumerate}
\end{theorem}
Note the result stated in the abstract easily follows from Theorem \ref{main}; see section  \ref{abs}.
We note that $P \in \Ass_R D$ if and only if $PA_P \in \Ass_{R_P}D_P$ if and only if $P\wh{A_P} \in \Ass_{\wh{R_P}}(D_P \otimes_{R_P} \wh{R_P})$.
Thus we may assume $R = K[[X_1,\ldots, X_d]]$ and $P = \m = (X_1, \ldots, X_d)$. As $\xb$ is a regular sequence in $R$ with $R/(\xb)$ regular it follows that $\xb$ is part of regular system of parameters of $R$. So we may assume without loss of generality that $x_i = X_i$ for $1\leq i \leq r$.

The proof of Theorem \ref{main} splits into two cases depending on characteristic of $k$.

\section{the case when $\charp k = p > 0$}
\s \label{setup-p} Let $K$ be a field of characteristic $p > 0$. Let $R = K[[X_1, \ldots, X_d]]$ and let $\xb = X_1, \ldots, X_r$. Let $\m = (X_1, \ldots, X_d)$. Set $G = G_\xb(R)$ be the associated graded ring of $R$ \wrt \ $\xb$. Note $G \cong R/(\xb)[X_1^*, \ldots, X_r^*]$ is a regular ring. Let $\Sc = R[(\xb)t, t^{-1}]$ be the extended Rees algebra of the ideal $(\xb)$. We note that $\Sc$ is a regular ring.  If $T$ is a regular ring containing a field $k$ of characteristic $p >0$ then the notion of $F_T$ modules and $F_T$-finite, $F_T$-modules was introduced by Lyubeznik, see \cite{Lyu-2}. We have $\Sc$ is a graded $F_\Sc$-finite, $F_\Sc$-module. Also note that
$\Sc_{t^{-1}} = R[t, t^{-1}]$ is also a graded $F_\Sc$-finite, $F_\Sc$-module. So $W = \Sc_{t^{-1}}/\Sc = \bigoplus_{n \geq 1}R/(\xb)^n$ is also a  graded $F_\Sc$-finite, $F_\Sc$-module.
Fix $m \geq 1$. Set $A_m = R/(\xb)^m$.
We show
\begin{theorem}\label{charp}(with hypotheses as in \ref{setup-p})
Let $I$ be an ideal in $R$
The following assertions are equivalent:
\begin{enumerate}[\rm (i)]
  \item  $\m \in \Ass_R H^i_I(A_1)$.
  \item $\m \in \Ass_R H^i_I(A_m)$.
\end{enumerate}
\end{theorem}
\begin{proof}
We may assume $K$ is infinite.
We note that $H^i_{I\Sc}(W) = \bigoplus_{n \geq 1} H^i_I(R/(\xb)^n)$ is a $F_\Sc$-finite, $F_\Sc$-module. By \cite[4.1]{P2} it follows that $$V = \ker(t^{-1} \colon H^i_{I\Sc}(W)(1) \rt H^i_{I\Sc}(W))$$
is a graded $F_G$-finite $F_G$-module. For $n \geq 0$ we have exact sequences:
\[
0 \rt V_n \rt H^i_I(R/(\xb)^{n+1}) \rt H^i_I(R/(\xb)^n).
\]
We also have $V_0 = H^i_I(R/(\xb))$. An easy induction yields
$$\Ass V_{n-1} \subseteq \Ass H^i_I(R/(\xb)^n) \subseteq \bigcup_{m = 0}^{n-1}\Ass V_m.$$

Let $\n = (X_{r+1}, \cdots, X_d)$ be maximal ideal of $G_0$. We note that $\m \in \Ass_R V_m$ if and only if  the Koszul homology module $H_{d - r}(\n, V_m) \neq 0$. Set $T = K[X_1^*, \ldots, X_r^*]$. By \cite[1.1]{P2} we have that $H_{d-r}(\n, V)$ is a  graded $F_T$-finite, $F_T$-module. In particular we have by  \cite[1.6(b)]{P2} that $H_{d-r}(\n, V)_m = H_{d-r}(\n, V_m)$ is non-zero for some $m \geq 0$ if and only if $ H_{d-r}(\n, V_j)$ is non-zero for ALL $j \geq 0$.

We now prove the result:

If $\m \in \Ass_R H^i_I(A_1)$ then as $V_0 = H^i_I(R/(\xb))$ we have $\m \in \Ass V_0$. So \\ $H_{d-r}(\n, V)_0 \neq 0$. It follows that $ H_{d-r}(\n, V_{m-1})$ is non-zero. So\\
$\m \in \Ass H^i_I(R/(\xb)^m)$.

Conversely if $\m \in \Ass_R H^i_I(A_m)$ then $\m \in \bigcup_{j = 0}^{m-1}\Ass V_j$. Say $\m \in \Ass V_l$. Then by above argument $\m \in \Ass V_j$ for all $j \geq 0$. In particular $\m \in \Ass V_0$ and as $V_0 = H^i_I(R/(\xb))$ the result follows.
\end{proof}

\section{the case when $\charp k =  0$}
\s \label{setup-0} Let $K$ be a field of characteristic  $0$. Let $R = K[[X_1, \ldots, X_d]]$ and let $\xb = X_1, \ldots, X_r$. Let $\m = (X_1, \ldots, X_d)$. Set $G = G_\xb(R)$ be the associated graded ring of $R$ \wrt \ $\xb$. Note $G \cong R/(\xb)[X_1^*, \ldots, X_r^*]$ is a regular ring. Let $\Sc = R[(\xb)t, t^{-1}]$ be the extended Rees algebra of the ideal $(\xb)$. We note that $\Sc$ is a regular ring. Let $D_0 = $ ring of $K$-linear differential operators on $G_0 = k[[X_{r+1}, \ldots, X_d]]$. Let $D $ be the $r^{th}$ Weyl algebra over $D_0$ with the obvious grading. Then note that $D$ is a ring of $D$-linear differential operators on $G$.
We note that $D$ is a ring of differential operators of $G$ and  $\Sc$. Also note that $t^{-1}$ commutes with the action of differential operators.
 We have $\Sc$ is a  $D$-module. Also note that
$\Sc_{t^{-1}} = R[t, t^{-1}]$ is also a $D$-module. So $W = \Sc_{t^{-1}}/\Sc = \bigoplus_{n \geq 1}R/(\xb)^n$ is also a  $D$-module.
Fix $m \geq 1$. Set $A_m = R/(\xb)^m$.

We show
\begin{theorem}\label{char0}(with hypotheses as in \ref{setup-0})
Let $I$ be an idealin $R$
The following assertions are equivalent:
\begin{enumerate}[\rm (i)]
  \item  $\m \in \Ass_R H^i_I(A_1)$.
  \item $\m \in \Ass_R H^i_I(A_m)$.
\end{enumerate}
\end{theorem}
\begin{proof}
We note that
$$0 \rt G \rt  W(1) \rt W \rt 0,$$
is a short exact sequence of $D$-modules. So the long exact sequence in cohomology $H^i_{I\Sc}(-)$ is a sequence graded of $D$-modules.
 As $t^{-1}$ commutes with action of $D$  it follows that $$V = \ker(t^{-1} \colon H^i_{I\Sc}(W)(1) \rt H^i_{I\Sc}(W))$$
is a graded $D$-module. As $V$ is a graded quotient of $H^i_I(G)$ it follows that $V$ is a
graded holonomic generalized Eulerian $D$-module.
 For $n \geq 0$ we have exact sequences:
\[
0 \rt V_n \rt H^i_I(R/(\xb)^{n+1}) \rt H^i_I(R/(\xb)^n).
\]
We also have $V_0 = H^i_I(R/(\xb))$. An easy induction yields
$$\Ass V_{n-1} \subseteq \Ass H^i_I(R/(\xb)^n) \subseteq \bigcup_{m = 0}^{n-1}\Ass V_m.$$

Let $\n = (X_{r+1}, \cdots, X_d)$ be maximal ideal of $G_0$. We note that $\m \in \Ass_R V_m$ if and only if the Koszul homology module  $H_{d - r}(\n, V_m) \neq 0$. Set $T = K[X_1^*, \ldots, X_r^*]$. By \cite[1.20]{P1}, we have that $H_{d-r}(\n, V)$ is a graded generalized Eulerian holonomic module over the Weyl algebra in $r$-variables over $K$. In particular we have by \cite[6.1]{P1}  that $H_{d-r}(\n, V)_m = H_{d-r}(\n, V_m)$ is non-zero for some $m \geq 0$ if and only if $ H_{d-r}(\n, V_j)$ is non-zero for ALL $j \geq 0$.

Rest of the argument is exactly similar as in \ref{charp}.
\end{proof}
\section{Proof of Result stated in abstract}\label{abs}
We now indicate a proof of the result stated in the abstract.
Let $I$ be the pre-image of $Q$ in $R$.  Let $\sharp E$ denote the number of elements in a set $E$.
We have by Theorem \ref{main}; $$ \sharp \Ass_{R/(\xb)^m)} H^i_Q(R/(\xb)^m) =  \sharp \Ass_R H^i_I(R/(\xb)^m) =
  \sharp \Ass_R H^i_I(R/(\xb)). $$

  We note that if $\charp k = p > 0$ then as $R/(\xb)$ is regular, $\Ass_R H^i_I(R/(\xb))$ is a finite set.

  If  $\charp k = 0$, $R$ is local or a smooth affine algebra over $k$, as $R/(\xb)$ is regular, $\Ass_R H^i_I(R/(\xb))$ is a finite set.

  The result follows.

\end{document}